\documentclass[a4paper, 11pt]{article}

\usepackage[
            includefoot,  %Uncomment to put page number above margin
            marginparwidth=0in,     % Length of section titles
            marginparsep=0in,       % Space between titles and text
            margin=1.45in,               % 1 inch margins
            includemp]{geometry}

\usepackage{bbm}
\usepackage{float}
\usepackage{graphicx}                  % derni\`ere \'etant la langue principale
\usepackage{amssymb}
\usepackage{amsfonts}
\RequirePackage{amsmath}
\RequirePackage{amsthm}

\RequirePackage{latexsym}
\RequirePackage{amsfonts}
\RequirePackage{amsmath}
\RequirePackage{amssymb}
\RequirePackage{amsthm}
\usepackage[latin1]{inputenc}

\newtheoremstyle{montheoreme}% name
  {}%				Space above
  {}% 			Space below
  {\itshape}%		Body font
  {}%				Indent amount (empty = no indent, \parindent = para indent)
  {\bf}%			Thm head font
  {.}%			Punctuation after thm head
  {.5em}%			Space after thm head: " " = normal interword space;
  {}%				Thm head spec (can be left empty, meaning `normal')
\newtheoremstyle{maremarque}% name
  {}%				Space above
  {}% 			Space below
  {}%				Body font
  {}%				Indent amount (empty = no indent, \parindent = para indent)
  {\bf}%			Thm head font
  {.}%			Punctuation after thm head
  {.5em}%			Space after thm head: " " = normal interword space;
  {}%				Thm head spec (can be left empty, meaning `normal')
\usepackage{fancyhdr}
\theoremstyle{montheoreme}
\newtheorem{exple}{Example}[section]
\newtheorem{thm}{Theorem}[section]
\newtheorem{defn}[thm]{Definition}
\newtheorem{prop}[thm]{Proposition}
\newtheorem{lem}[thm]{Lemma}
\newtheorem{rque}{Remark} [section]

\theoremstyle{maremarque}
\newtheorem*{rmq}{Remark}{}

\DeclareMathOperator{\osc}{osc}

\DeclareMathOperator{\cov}{cov}
\DeclareMathOperator{\Ent}{Ent}

\newcommand{\R}{\mathbb{R}}
\newcommand{\N}{\mathbb{N}}

\makeatletter

\@addtoreset{equation}{chapter}
\makeatother

\makeatletter

\@addtoreset{equation}{section}
\makeatother

\begin{document}

\title{Modified logarithmic Sobolev inequalities for canonical ensembles.}
\author{Max Fathi \thanks{LPMA, University Paris 6, France, max.fathi@etu.upmc.fr.} }
\date{\today}

\maketitle

\begin{abstract}
In this paper, we prove modified logarithmic Sobolev inequalities for canonical ensembles with superquadratic single-site potential. These inequalities were introduced by Bobkov and Ledoux, and are closely related to concentration of measure and transport-entropy inequalities. Our method is an adaptation of the iterated two-scale approach that was developed by Menz and Otto to prove the usual logarithmic Sobolev inequality in this context. As a consequence, we obtain convergence in Wasserstein distance $W_p$ for Kawasaki dynamics on the Ginzburg-Landau model.

\bigskip

{\begin{footnotesize}
%\noindent\emph{MSC:} ? \newline
  \end{footnotesize}
}

\end{abstract}

\tableofcontents

\vspace{1cm}

{\Large \textbf{Introduction}}

The logarithmic Sobolev inequality is an inequality allowing to embed the Sobolev space $H^1(\mu)$ in the Orlicz space
$L^2\log L(\mu)$, just like the usual Sobolev inequalities embed $H^1$ in $L^p$ for some $p > 2$. It was introduced by Gross in \cite{Gr}, and has been shown to be very useful in some problems of statistical physics, such as long-time convergence to equilibrium, and hydrodynamic limits (see for example \cite{GOVW}).

One case of measures where such an inequality has been useful is for canonical ensembles, which are probability measures $\mu(dx) = \exp(-\sum \psi(x_i))$ on the hyperplane $\{\sum x_i = Nm\}$ of $\R^N$. In the recent contribution \cite{MO}, Menz and Otto proved that, if the function $\psi$ is a bounded perturbation of a uniformly convex function, then the canonical ensemble satisfies a logarithmic Sobolev inequality, with a constant independent of the mean $m$ and the dimension $N$. 

The result of \cite{MO} covers potentials which behave like $|x|^p$ for some $p \geq 2$. A natural question is whether we can improve the LSI when $p$ is strictly larger than 2. For this purpose, we investigate whether a variant of the LSI called the modified logarithmic Sobolev inequality, which was introduced by Bobkov and Ledoux in \cite{BL}, is satisfied by canonical ensembles. Our method is a generalization of the iterated two-scale approach that was used in \cite{MO} to obtain the usual LSI.

\vspace{1cm}

{\Large\textbf{Notations}}

\begin{itemize}

\item
$p$ will always denote a real number satisfying $p \geq 2$, and $q$ will always be the dual exponent of $p$, that is the only real number satisfying $\frac{1}{p} + \frac{1}{q} = 1.$

\item
We denote by $||\cdot||_p$ the usual $\ell^p$ norm on $\R^N$, and $\langle \cdot, \cdot \rangle$ the scalar product associated to the $\ell^2$ norm. 

\item 
When $X$ is an affine subspace of $\mathbb{R}^N$ and $f : X \rightarrow \mathbb{R}$ is a smooth function, we define the gradient of $f$ at point $x$ by $(\nabla f)_i(x) := \frac{\partial f}{\partial x_i}(x),$ where the function $f$ has been extended to be constant in the direction normal (for the $L^2$ structure) to $X$ in $\mathbb{R}^N$. This definition coincides with the usual one.

\item
$Z$ is a constant enforcing unit mass for a probability measure.

\item 
$C$ is a positive constant, which may change from line to line, or even within a line.

\item
$\operatorname{Ent}_{\mu}(f) := \int{f \log f d\mu} - \left(\int{fd\mu}\right) \log \int{fd\mu}$ is the entropy of the (nonnegative) function $f$ with respect to the probability measure $\mu$.

\item $P^t$ is the adjoint of the linear operator $P$.

\item $\mathcal{L}^N$ is the $N$-dimensional Lebesgue measure.

\end{itemize}

\section{Background and Main Results}

In this paper, we are interested in the following family of inequalities, which generalizes the logarithmic Sobolev inequality.

\begin{defn}
A probability measure $\mu$ satisfies a p-modified logarithmic Sobolev inequality with parameter $\rho$ if, for all positive compactly supported $C^1$ function f, we have
\begin{equation}
\operatorname{Ent}_{\mu}(f) \leq \frac{1}{\rho}\int{\frac{||\nabla f||_q^q}{f^{q-1}}d\mu},
\end{equation}
where $q$ is the dual exponent of $p$, that is
$$\frac{1}{p} + \frac{1}{q} = 1.$$
Equivalently, $\mu$ satisfies this inequality if for any such function $f$, we have
\begin{equation} \label{mlsi2}
\operatorname{Ent}_{\mu}(f^q) \leq \frac{q^q}{\rho}\int{||\nabla f||_q^qd\mu}.
\end{equation}
\end{defn}

In the case $p = 2$, this is the usual logarithmic Sobolev inequality. Many results on these inequalities can be found in \cite{BZ}, and we recall some of them in the sequel. It is well known that the usual LSI implies Gaussian concentration properties. In the same way, modified logarithmic Sobolev inequalities are linked to the following form of concentration of measure :

\begin{defn}
A probability measure $\mu$ on a metric space $(X,d)$ has the p-exponential concentration property with parameter $c$ if, for any 1-Lipschitz function $f : X \rightarrow \mathbb{R}$ and every $r \geq 0$, we have 
$$\mu \left( f \geq \int{fd\mu} + r \right) \leq \exp \left( -\frac{c t^p}{p(p-1)^{p-1}}\right).$$
\end{defn}

\begin{thm} \label{LTC}
If $\mu$ satisfies $p-$LSI($\rho$), then $\mu$ satisfies p-exponential concentration for the $\ell^p$ distance
\end{thm}

We refer to [BZ, Theorem 1.3] for a proof of this result.

We consider a (periodic) lattice spin system of N continuous variables governed by a Ginzburg-Landau type potential $\psi : \R \rightarrow \R$. The grand canonical measure on $\R^N$ has density

\begin{equation} 
\frac{d\mu_{N}}{d\mathcal{L}^N}(x) = \frac{1}{Z}\exp \left( - \underset{i = 1}{\stackrel{N}{\sum}} \hspace{1mm} \psi(x_i) \right).
\end{equation}

We shall assume that the potential $\psi$ is of class $C^1$ and is of the form
\begin{equation} \label{assumption_potential}
\psi(x) = \psi_c(x) + \delta\psi(x); \hspace{5mm} \psi_c''(x) \geq c(1 + |x|^{p-2}); \hspace{5mm} ||\delta\psi||_{\infty} + ||\delta\psi'||_{\infty}  < +\infty.
\end{equation}
Under these assumptions, $\psi_c$ is a uniformly p-convex and uniformly convex function. A typical example would be the quartic double-well potential $\psi(x) = (x^2-1)^2$. For a definition of p-convexity see Theorem \ref{be}

\begin{rmq}
Our results are still valid if we only ask $\psi_c$ to satisfy $\psi_c''(x) \geq c(1 + |x-x_0|^{p-2})$ for some $x_0$. The proof is exactly the same, but the extra assumption makes the calculations easier to read.
\end{rmq}

To simplify notations, we define the Hamiltonian 

\begin{equation}
H(x) := \underset{i = 1}{\stackrel{N}{\sum}} \hspace{1mm} \psi(x_i) +  \log Z,
\end{equation}
so that $\mu(dx) = \exp(-H(x))dx$.

We will add to the situation a constraint of fixed mean spin. The phase state space is
 \begin{equation*}
    \label{e_definition_of_X_N_M}
   X_{N,m} := \left\{  x \in \R^N, \ \frac{1}{N} \sum_{i=1}^N x_i =m   \right\},
 \end{equation*}
 
where $m$ is an arbitrary real number. This space is a hyperplane of $\R^N$ with a fixed mean constraint. We endow this space with the $\ell^2$ inner product
\begin{equation}
\langle x, \tilde{x} \rangle_X = \underset{i = 1}{\stackrel{N}{\sum}} \hspace{1mm} x_i \tilde{x}_i.
\end{equation}

For a given $m \in \R$, we consider the restriction $\mu_{N,m}$ of the grand canonical measure to $X_{N,m}$, that is
\begin{equation} \label{def_grand_can}
\frac{d\mu_{N,m}}{d\mathcal{L}^{N-1}}(x) = \frac{1}{Z}\mathbbm{1}_{(1/N)\sum x_i = m}\exp \left( - \underset{i = 1}{\stackrel{N}{\sum}} \hspace{1mm} \psi(x_i) \right).
\end{equation}

This measure is called the canonical ensemble. It gives the distribution of the random variables $x_i$ conditioned on the event that their mean value is given by $m$.

It was shown in \cite{MO} that when the single site potential satisfies assumption (\ref{assumption_potential}) with $p = 2$, then the canonical ensemble satisfies the classical logarithmic Sobolev inequality for some constant $\rho > 0$ that is independent of both $m$ and $N$. Our aim in this paper is to generalize this result for the modified LSI, and we obtain the following : 

\begin{thm} \label{main_thm}
Under the assumption (\ref{assumption_potential}), the canonical ensemble $\mu_{N,m}$ satisfies p-LSI($\rho$) for some constant $\rho > 0$ that is independent of both $N$ and $m$.
\end{thm}

The proof in \cite{MO} uses a method called the iterated two-scale approach, which generalizes a method that was developed in \cite{GOVW}. The idea is to use a decomposition of the system into a macroscopic component and a fluctuations component, obtained by coarse-graining. There are then two main ideas: then first is to prove that if the laws of both the macrscopic and fluctuations part satisfy the desired functional inequality, then the law of the full system also satisfies the inequality. The second idea is tho show that, if we iterate this decomposition often enough for the successive macroscopic component, then we obtain additional convexity properties, which allow us to prove that the macroscopic component satisfies the inequality we are looking for.

Our proof here follows the iterated two-scale approach, but uses several new ingredients : 

\begin{itemize}

\item To deduce the modified LSI for the full measure from the inequality for the macroscopic measure, we use the $L^1$ Poincar\'e inequality to bound a crucial covariance term;

\item In addition to uniform convexity, we must prove uniform p-convexity for the macroscopic Hamiltonian, as soon as we have coarse-grained the system often enough;

\item We use the Prekopa-Leindler inequality to show that, if the single-site potential satisfies assumption (\ref{assumption_potential}), then the coarse-grained potential also does.
\end{itemize}

It was shown in \cite{OV} (and then in \cite{BGL} and \cite{Go} with alternative proofs) that the classical logarithmic Sobolev inequality implies that the square root of the entropy controls the Wasserstein distance of order two (up to a multiplicative constant). Such an inequality is known as Talagrand's inequality. Similarly, we can define a class of inequalities which generalizes the Talagrand inequality to Wasserstein distances of order $p$, which is linked to the modified logarithmic Sobolev inequality we just defined.

\begin{defn}
A probability measure $\mu$ satisfies a Talgrand inequality with parameter $p$ and constant $\rho$ if, for any probability measure $\nu$, we have
$$W_p^p(\mu, \nu) \leq \frac{p}{\rho}\operatorname{Ent}_{\mu}(\nu).$$
We will denote this inequality by $T_p(\rho)$.
\end{defn}

\begin{rque}
Some people define $T_p(\rho)$ as $W_p(\mu, \nu)\leq \sqrt{\frac{2}{\rho}\operatorname{Ent}_{\mu}(\nu)}.$ These two definitions are \textbf{not} equivalent.
\end{rque}

It was shown by Marton in [M] that transport-entropy inequalities such as Talagrand inequalities imply concentration properties. These inequalities are also linked to modified logarithmic Sobolev inequalities through the following result, which was proven in \cite{GRS} :

\begin{prop}
If $\mu$ satisfies p-LSI($\rho$), then it satisfies $T_p(\tilde{\rho})$, with constant $\tilde{\rho} = ((p-1)\rho)^{p-1}$ and the $\ell^p$ distance.
\end{prop}

Combining this Proposition and Theorem \ref{main_thm}, we obtain

\begin{thm} \label{thm_tal}
Under the assumption (\ref{assumption_potential}), the canonical ensemble $\mu_{N,m}$ satisfies $T_p(\tilde{\rho})$ for some constant $\tilde{\rho} > 0$ that is independent of $N$ and $m$.
\end{thm}

In section 3, an application of these modified LSI is presented, to obtain rates of convergence in the Wasserstein distance $W_p$ for for the Kawasaki dynamic on the Ginzburg-Landau model. 

These inequalities can also be used to obtain quantitative rates on the speed of convergence to the hydrodynamic limit in $W_p$ of Kawasaki dynamics, in conjunction with the results in \cite{F} on convergence in relative entropy.

\section{The iterated two-scale approach for modified logarithmic Sobolev inequalities}

In this section, we shall prove Theorem \ref{main_thm}. The proof is based on a coarse-graining argument. The coarse-graining operator we shall use is defined as follows : Assume $N = 2^K$ for some large $K \in \N$. We define $P : X_{N,m} \rightarrow X_{N/2, m}$ by 
\begin{equation}
P(x_1, x_2, ..., x_N) := \left( \frac{x_1 + x_2}{2}, \frac{x_3 + x_4}{2}, ..., \frac{x_{N-1} + x_N}{2} \right).
\end{equation}
Using this operator, we can decompose $\mu_{N,m}$ as 
$$\mu_{N,m}(dx) = \mu(dx|y) \bar{\mu}(dy)$$
where $\bar{\mu}$ is the push forward of $\mu$ under $P$ and $\mu(dx|y)$ is the conditional measure of $x$ given $Px = y$. 

The key element of the iterated two-scale approach of \cite{MO} is that, when the coarse-grained measure $\bar{\mu}$ satisfies a logarithmic Sobolev inequality, the full measure $\mu$ also does. We shall prove the same result for modified logarithmic Sobolev inequalities : 

\begin{prop} \label{hierarchic_criterion}
If $\bar{\mu}$ satisfies p-LSI($\rho$) with $\rho$ independent of $N$ and $m$, then $\mu_{N,m}$ satisfies p-LSI($\tilde{\rho}$) with $\tilde{\rho}$ also independent of $N$ and $m$.
\end{prop}

To prove Theorem \ref{main_thm}, we shall iteratively apply Proposition \ref{hierarchic_criterion}. To be able to do so, we need to show that the coarse-grained measure has the same form as the original measure, \textit{i.e.} that it has the structure $\exp(-\sum \tilde{\psi}(y_i))$ with $\tilde{\psi}$ a bounded perurbation of a p-convex and uniformly convex function. To do this, lets look at the structure of $\bar{\mu}$. We have

$$\bar{\mu}(dy) = \frac{1}{Z}\exp\left( - 2\underset{i = 1}{\stackrel{N/2}{\sum}} \hspace{1mm} \textit{R}\psi(y_i) \right)dy$$
where 
\begin{equation}
\textit{R}\psi(y) := -\frac{1}{2}\log \left(\int_{\R}{\exp(-\psi(x + y) - \psi(-x + y))dy}\right)
\end{equation}
is the renormalized single-site potential. We denote by $\textit{R}^M\psi$ the M-times renormalized single-site potential. We then have the following result : 

\begin{lem} \label{structure}
If $\psi = \psi_c + \delta\psi$ is a bounded perturbation of a p-convex, uniformly convex potential, then $\textit{R}\psi$ also is.
\end{lem}

The last element of the proof is that, after a large but finite number of coarse-grainings, the measure we obtain will be uniformly p-convex, and therefore satisfy p-LSI($\rho$) for some $\rho > 0$. This convexification phenomenon is well-known in statistical physics, as a consequence of the equivalence of ensembles principle. We state is as the following lemma : 

\begin{lem} \label{lem_convexification}
Let $\psi$ be a a bounded perturbation of a p-convex, uniformly convex potential. Then there is an integer $M_0$ such that for all
$M \geq M_0$ the M-times renormalized single-site potential $\textit{R}^M\psi$ is uniformly p-convex with constant $\rho$ independent of the system size $N$, $M$ and of the mean $m$.
\end{lem}

The proof of Theorem \ref{main_thm} is a direct consequence of these three results : we just have to iterate Proposition \ref{hierarchic_criterion} a large, but finite, number of times. Lemma \ref{structure} guarantees that this iteration is legitimate, while Lemma \ref{lem_convexification} tells us that after a finite number of coarse-grainings, the macroscopic measure we obtain is uniformly p-convex, and therefore satisfies p-LSI($\rho$) for some $\rho$ independent of $N$ and $m$. Since Proposition \ref{hierarchic_criterion} allows us to deduce the inequality for the microscopic measure as long as the coarse-grained measure also satisfies it, we can conclude that the original measure $\mu_{N,m}$ satisfies p-LSI($\rho$) for some constant $\rho > 0$ that is independent of both $N$ and $m$. So all that remains is to prove these three results.

\begin{proof} [Proof of Proposition \ref{hierarchic_criterion}]
First we use the decomposition
\begin{equation} \label{dec_ent}
\operatorname{Ent}_{\mu}(f) = \operatorname{Ent}_{\bar{\mu}}(\bar{f}) + \int{\operatorname{Ent}_{\mu(\cdot|y)}(f)\bar{\mu}(dy)},
\end{equation}
which can easily be verified through conditioning. We will then bound the two terms on the right-hand side of (\ref{dec_ent}) by using modified LSI for the measures $\mu(dx|y)$ and $\bar{\mu}$.

\begin{lem} \label{lem_ineq_micro}
There exists $\lambda > 0$ such that $\mu(dx|y)$ satisfies p-LSI($\lambda$) for all $y \in Y$.
\end{lem}

\begin{proof}[Proof of Lemma \ref{lem_ineq_micro}]
Since $\mu(dx|y) = \bigotimes \mu_{2,y_i}(dx_{2i-1}, dx_{2i})$, by the tensorization property (see Proposition \ref{tensor}), we just have to show that  $\mu_{2,m}$ satisfies p-LSI($\lambda$) for some $\lambda > 0$ which does not depend on the real number $m$. 

We have 
\begin{align}
\mu_{2,m}(dx_1, dx_2) &= \frac{1}{Z} \mathbbm{1}_{x_1 + x_2 = 2m}\exp(-\psi(x_1) - \psi(x_2))dx \notag \\
&= \frac{1}{Z} \mathbbm{1}_{x_1 + x_2 = 2m}\exp(-\psi_c(x_1) - \psi_c(x_2) - \delta\psi(x_1) - \delta\psi(x_2))dx \notag 
\end{align}
It is immediate that $(x_1, x_2) \rightarrow \psi_c(x_1) + \psi_c(x_2)$ is uniformly p-convex, so an application of Theorem \ref{be} yields that the measure $\tilde{\mu}(dx) = Z^{-1} \mathbbm{1}_{x_1 + x_2 = 2m}\exp(-\psi_c(x_1) - \psi_c(x_2))dx$ satisfies p-LSI($\tilde{\lambda}$) for some $\tilde{\lambda} > 0$ which doesn't depend on $m$. Since $\delta\psi$ is bounded, $\mu_{2,m}$ is a bounded perturbation of $\tilde{\mu}$, and we immediately deduce from Proposition \ref{hol_str} that it satisfies p-LSI($\lambda$) for some $\lambda > 0$ which does not depend on $m$. This concludes the proof of Lemma \ref{lem_ineq_micro}
\end{proof}

We can now continue the proof of Proposition \ref{hierarchic_criterion}. As a consequence of Lemma \ref{lem_ineq_micro}, we have
\begin{align} \label{dec_ent2}
\int{\operatorname{Ent}_{\mu(\cdot|y)}(f)\bar{\mu}(dy)} &\leq \int_Y{\frac{1}{\lambda}\int_{ \{Px = y\}}{\frac{|(id_X - 2P^tP)\nabla f|_q^q}{f^{q-1}} \mu(dx|y)} \bar{\mu}(dy)} \notag \\
&= \frac{1}{\lambda}\int_X{\frac{|(id_X - 2P^tP)\nabla f|_q^q}{f^{q-1}} \mu(dx)}.
\end{align}

By assumption, $\bar{\mu}$ satisfies p-LSI($\rho$), so that
\begin{equation} \label{ent_micro}
\operatorname{Ent}_{\bar{\mu}}(\bar{f}) \leq \frac{1}{\rho}\int_Y{\frac{|\nabla_Y \bar{f}|_q^q}{\bar{f}^{q-1}}\bar{\mu}(dy)}
\end{equation}

To deduce from this inequality a bound on the macroscopic entropy by a function of the microscopic gradient, we need to relate $\nabla_Y \bar{f}$ and $\nabla f$. This is the point of the following lemma :

\begin{lem} \label{lien_grad}
$$\nabla_Y \bar{f}(y) = 2P\int{\nabla f(x)\mu(dx|y)} + 2P\cov_{\mu(dx|y)}(f, \nabla H).$$
\end{lem}

This lemma was already used for the same reasons in \cite{GOVW} and \cite{MO}. For now, we defer its proof. Using this result, the convexity of $(x, b) \rightarrow ||x||^q_q / b^{q-1}$ and the inequality $|a + b|^q \leq C(q)(|a|^q + |b|^q)$, we get
\begin{align} \label{ent_macr}
\operatorname{Ent}_{\bar{\mu}}&(\bar{f}) \leq \frac{1}{\rho} \int{\frac{|\nabla \bar{f}|_q^q}{\bar{f}^{q-1}}\bar{\mu}(dy)} \notag \\
&= \frac{1}{\rho} \int{\frac{\left|2P\int{\nabla f(x)\mu(dx|y)} + 2P\cov_{\mu(dx|y)}(f, \nabla H) \right|_q^q}{\left(\int{f(x)\mu(dx|y)}\right)^{q-1}}\bar{\mu}(dy)} \notag \\
&\leq \frac{C}{\rho}\int_X{\frac{|2P\nabla f(x)|_q^q}{f^{q-1}}\mu(dx)} + \frac{C}{\rho} \int{\frac{|2P\cov_{\mu(dx|y)}(f, \nabla H) |_q^q}{\bar{f}}\bar{\mu}(dy)} 
\end{align}

We have
\begin{align} \label{dec_psi}
|2P&\cov_{\mu(dx|y)}(f, \nabla H) |_q^q = \underset{i = 1}{\stackrel{N/2}{\sum}} \hspace{1mm} \int{|\cov_{\mu_{2,y_i}}(f, (2P\nabla H)_i)|^q \underset{j \neq i}{\bigotimes} \mu_{2,y_j}(dx_{2j-1}, dx_{2j})}  \notag \\
&= \underset{i = 1}{\stackrel{N/2}{\sum}} \hspace{1mm} \int{|\cov_{\mu_{2,y_i}}(f, \psi'(x_{2i-1}) + \psi'(x_{2i}))|^q \underset{j \neq i}{\bigotimes} \mu_{2,y_j}(dx_{2j-1}, dx_{2j})}  \notag \\
&\leq C(q)\underset{i = 1}{\stackrel{N/2}{\sum}} \hspace{1mm} \int{|\cov_{\mu_{2,y_i}}(f, \psi_c'(x_{2i-1}) + \psi_c'(x_{2i}))|^q \underset{j \neq i}{\bigotimes} \mu_{2,y_j}(dx_{2j-1}, dx_{2j})}  \notag \\
& \hspace{5mm} + C(q)\underset{i = 1}{\stackrel{N/2}{\sum}} \hspace{1mm} \int{|\cov_{\mu_{2,y_i}}(f, \delta\psi'(x_{2i-1}) + \delta\psi'(x_{2i}))|^q\underset{j \neq i}{\bigotimes} \mu_{2,y_j}(dx_{2j-1}, dx_{2j})} 
\end{align}

To bound the first part term, we use the following inequality, due to \cite{MO} : 

\begin{lem} [Asymmetric Brascamp-Lieb inequality] \label{as_bl}
Let  $\nu(dx) = \frac{1}{Z}\exp(-\psi(x))dx$ a probability measure on $\R$, where $\psi = \psi_c + \delta\psi$ is a bounded perturbation of a strictly convex potential. Then for any functions $f$ and $g$, we have
$$|\cov_{\nu}(f,g)| \leq \exp(-3\osc \delta \psi) \hspace{1mm} \underset{x}{\sup} \hspace{1mm} \left| \frac{g'(x)}{\psi_c''(x)} \right| \int{|f'|d\nu}.$$
\end{lem}

Using this lemma, we get
\begin{align}
\int&{|\cov_{\mu_{2,y_i}}(f, \psi_c'(x_{2i-1}) + \psi_c'(x_{2i}))|^q \underset{j \neq i}{\bigotimes} \mu_{2,y_j}(dx_{2j-1}, dx_{2j})} \notag \\
&\leq C\int{\left(\int{\left| \frac{df}{dx_{2i-1}} \right| + \left| \frac{df}{dx_{2i}} \right| \mu_{2,y_i}(dx_{2i-1}, dx_{2i})}\right)^q\underset{j \neq i}{\bigotimes} \mu_{2,y_j}(dx_{2j-1}, dx_{2j})} \notag \\
&\leq \left(\int{\left(\int{f(x)\mu_{2,y_i}(dx_{2i-1}, dx_{2i})}\right)^{q/p}\underset{j \neq i}{\bigotimes} \mu_{2,y_j}(dx_{2j-1}, dx_{2j})}\right)  \notag \\
&\hspace{5mm} \times \left(\int{\left(\int{\frac{|\frac{df}{dx_{2i-1}}|^q + |\frac{df}{dx_{2i}}|^q}{f^{q-1}}\mu_{2,y_i}(dx_{2i-1}, dx_{2i})}\right)\underset{j \neq i}{\bigotimes} \mu_{2,y_j}(dx_{2j-1}, dx_{2j})}\right)  \notag \\
&\leq C \bar{f}(y)^{q-1}\left(\int{\frac{|\frac{df}{dx_{2i-1}}|^q + |\frac{df}{dx_{2i}}|^q}{f^{q-1}} \mu(dx|y)}\right),
\end{align}
where the last inequality uses the fact that $q/p = q-1 \leq 1$, and therefore $a \longrightarrow a^{q-1}$ is concave. Summing up, we obtain
\begin{align} \label{borne_cov1}
\underset{i = 1}{\stackrel{N/2}{\sum}} \hspace{1mm} \int&{|\cov_{\mu_{2,y_i}}(f, \psi_c'(x_{2i-1}) + \psi_c'(x_{2i}))|^q \underset{j \neq i}{\bigotimes} \mu_{2,y_j}(dx_{2j-1}, dx_{2j})} \notag \\
&\leq C\bar{f}(y)^{q-1}\int{\frac{|\nabla f|_q^q}{f^{q-1}} \mu(dx|y)}.
\end{align}

For the second part of (\ref{dec_psi}), we use the following $L^1$ Poincar\'e inequality, which is Proposition 1.8 of \cite{L2} : 

\begin{thm} \label{l1_poincare}
Consider a measure $\mu = \exp(-H)dx$ on $\R^d$, and assume that $H$ is a bounded perturbation of a uniformly convex potential. Then there exists a constant $\alpha > 0$ such that, for any smooth function $f$, we have
$$\int{\left|f(x) - \int{f(y)\mu(dy)}\right|\mu(dx)} \leq \alpha \int{|\nabla f(x)|\mu(dx)}.$$
\end{thm}

Since $\delta \psi'$ is bounded, we have
\begin{align}
&|\cov_{\mu_{2,y_i}}(f, \delta\psi'(x_{2i-1}) + \delta\psi'(x_{2i}))|^q \notag \\
& \hspace{4mm} \leq (2||\delta\psi'||_{\infty})^q\left(\int{\left|f(x) - \int{f d\mu_{2,y_i}}\right|\mu_{2,y_i}(dx_{2i-1}, dx_{2i})}\right)^q \notag \\
& \hspace{4mm} \leq C\left(\int{|\frac{df}{dx_{2i-1}}| + |\frac{df}{dx_{2i}}| \mu_{2,y_i}(dx_{2i-1}, dx_{2i})}\right)^q \notag \\
& \hspace{4mm} \leq C\left(\int{f(x)\mu_{2,y_i}(dx_{2i-1}, dx_{2i})}\right)^{q-1}\int{\frac{|\frac{df}{dx_{2i-1}}|^q + |\frac{df}{dx_{2i}}|^q}{f(x)^{q-1}} \mu_{2,y_i}(dx_{2i-1}, dx_{2i})},
\end{align}
where we have used Theorem \ref{l1_poincare} and the convexity of the function $(a,b) \rightarrow a^q/b^{q-1}$.

With the previous two bounds, we get

\begin{equation} \label{borne_cov2}
\int{\frac{|2P\cov_{\mu(dx|y)}(f, \nabla H) |_q^q}{\bar{f}}\bar{\mu}(dy)} \leq C \int{\frac{|\nabla f|_q^q}{f^{q-1}} \mu(dx)}.
\end{equation}

We then state the elementary inequalities

$$|2Px|_q^q = \underset{i}{\sum} \hspace{1mm} |x_{2i-1} + x_{2i}|^q \leq C(q)\underset{j}{\sum} \hspace{1mm} |x_j|^q = C(q)|x|_q^q$$
and
$$|(id - 2P^tP)x|_q^q = \underset{i}{\sum} \hspace{1mm} |\frac{x_{2i-1} - x_{2i}}{2}|^q + |\frac{x_{2i} - x_{2i-1}}{2}|^q \leq \frac{C(q)}{2^q}|x|_q^q.$$

Using these bounds, (\ref{dec_ent2}), (\ref{ent_macr})  and (\ref{borne_cov2}), we get Proposition \ref{hierarchic_criterion}.
\end{proof}

Before we move on to the proofs of Lemmas \ref{structure} and \ref{lem_convexification}, here is a short proof of Lemma \ref{lien_grad}, which is taken from \cite{GOVW}.

\begin{proof} [Proof of Lemma \ref{lien_grad}]
Recall that
\begin{align}
\bar{f}(y) &= \int_{\{Px = y\}}{f(x)\mu(dx|y)} \notag \\
&= \frac{1}{\int_{\{Px = 0\}}{\exp(-H(2P^ty + z))dz}}\int_{\{Px = 0\}}{f(2P^ty + z)\exp(-H(2P^ty + z))dz}, \notag 
\end{align}
and therefore, for any $\tilde{y} \in Y$, we have
\begin{align}
\nabla_Y \bar{f}(y) \cdot \tilde{y} &= 2\int{\nabla f(x) \cdot P^t \tilde{y} \mu(dx|y)} - 2\int{f(x) \nabla H(x) \cdot P^t \tilde{y} \mu(dx|y)} \notag \\
&\hspace{5mm} - 2 \left(\int{f(x)\mu(dx|y)} \right)\left(\int{-H(x)\cdot P^t\tilde{y} \mu(dx|y)} \right) \notag \\
&= 2\int{P\nabla f(x) \cdot \tilde{y} \mu(dx|y)} - 2\int{f(x) P\nabla H(x) \cdot \tilde{y} \mu(dx|y)}\notag \\
&\hspace{5mm} + 2 \left(\int{f(x)\mu(dx|y)} \right)\left(\int{PH(x)\cdot \tilde{y} \mu(dx|y)} \right), \notag 
\end{align}
which is what we wanted to prove.
\end{proof}

We are now done with the proof of Proposition \ref{hierarchic_criterion}. The next step is to prove Lemma \ref{structure} : 

\begin{proof} [Proof of Lemma \ref{structure}]
We define 
$$\bar{\psi}_c(m) := -\frac{1}{2}\log \int{\exp(-\psi_c(m + x) - \psi_c(m - x))dx}$$
and
\begin{align}
\bar{\delta\psi}(m) &:=  -\frac{1}{2}\log \int_{\R}{\exp(-\psi(m + x) - \psi(m - x))dx} \notag \\
& \hspace{5mm} + \frac{1}{2}\log \int{\exp(-\psi_c(m + x) - \psi_c(m - x))dx}.
\end{align}

Our aim is to show that $\bar{\delta\psi}$ is bounded in the $C^1$ topology, and that $\bar{\psi}_c$ is uniformly convex and p-convex. Since $\textit{R}\psi = \bar{\psi}_c + \bar{\delta\psi}$, this will show that $\bar{\mu}$ has the desired structure.

The fact that $\bar{\psi}_c$ is uniformly convex has been done in \cite{MO}, using the (symmetric) Brascamp-Lieb inequality. Here we also need to prove that $\bar{\psi}_c$ is uniformly p-convex. To do this, we shall use the Prekopa-Leindler inequality, and the same method will also show that $\bar{\psi}_c$ is uniformly convex (which is not surprising, since the Pr\'ekopa-Leindler inequality is stronger than the Brascamp-Lieb inequality, as was shown in \cite{BL}).

\begin{thm}
Let $t \in (0,1)$ and $f,g,h$ be non-negative measurable functions defined on $\R$. Suppose that these functions satisfy
$$h(tx + (1-t)y) \geq f(x)^tg(y)^{1-t}$$
for all $x$ and $y$ in $\R$. Then 
$$\int{h(x)dx} \geq \left(\int{f(x)dx} \right)^t \left(\int{g(x)dx} \right)^{1-t}.$$
\end{thm}

Let $h(x,m) = \exp\left( - \psi_c(x+m) - \psi_c(-x + m) \right)$. We have for any $t \in (0,1)$

\begin{align}
&h(tx+(1-t)y, tm + (1-t)m') \notag \\
&= \exp\left( - \psi_c(tx+(1-t)y+tm + (1-t)m') - \psi_c(-tx -(1-t)y + tm + (1-t)m') \right) \notag \\
&\geq \exp\left( - t\psi_c(x + m) - t\psi_c(-x + m) -(1-t)\psi_c(y + m') - (1-t)\psi_c(-y + m')\right) \notag \\
& \hspace{1cm} \times \exp\left( ct(1-t)|m - m' + x - y|^p + ct(1-t)|m - m' + y - x|^p \right) \notag \\
&\geq \exp\left( - t(\psi_c(x + m) + \psi_c(-x + m)) -(1-t)(\psi_c(y + m') + \psi_c(-y + m')) + 2ct(1-t)|m - m'|^p \right) \notag \\
&= \exp\left( - \psi_c(x + m) - \psi_c(-x + m)+ c(1-t)|m - m'|^p \right)^t \notag \\
& \hspace{1cm} \times \exp\left( -\psi_c(y + m') - \psi_c(-y + m') + ct|m - m'|^p \right)^{1-t}
\end{align}
Applying the Pr\'ekopa-Leindler inequality with 
$$h(x) = h(x, tm + (1-t)m'),$$
$$f(x) = \exp\left( - \psi_c(x + m) - \psi_c(-x + m)+ c(1-t)|m - m'|^p \right)$$
and 
$$g(x) = \exp\left( -\psi_c(y + m') - \psi_c(-y + m') + ct|m - m'|^p \right)$$ 
then yields
\begin{align}
\int_{\R}&{\exp\left( - \psi_c(x + tm + (1-t)m') - \psi_c(- x  + tm + (1-t)m') \right)dx} \notag \\
& \hspace{5mm} \geq \left(\int_{\R}{\exp\left( - \psi_c(x + m) - \psi_c(- x  + m) + c(1-t)|m - m'|^p \right)dx}\right)^t \notag \\
& \hspace{15mm} \times \left(\int_{\R}{\exp\left( - \psi_c(x + m') - \psi_c(- x  + m') + ct|m - m'|^p \right)dx}\right)^{1-t}
\end{align}
so that 
$$\bar{\psi}_c(tm + (1-t)m') \leq t\bar{\psi}_c(m) + (1-t)\bar{\psi}_c(m') - ct(1-t)|m - m'|^p,$$
which is the inequality we were aiming for.

The same arguments, applied with $p = 2$ also show that $\bar{\psi}_c$ inherits uniform convexity from $\psi_c$.

We still need to prove bounds on $\bar{\delta\psi}$ and its first derivative. These were already proven in \cite{MO}, we reproduce their argument here.

It will be convenient to introduce the probability measures 
$$\nu(dx) = \frac{1}{Z}\exp(-\psi(-x + m) - \psi(x + m))dx$$
and 
$$\nu_c(dx) = \frac{1}{Z}\exp(-\psi_c(-x + m) - \psi_c(x + m))dx,$$
so that we have
$$\bar{\delta\psi} = -\frac{1}{2} \log \int{\exp(-\delta\psi(-x + m) - \delta\psi(x + m))\nu_c(dx)}$$
and the bound $||\bar{\delta\psi} ||_{\infty} < \infty$ immediately follows from $||\delta\psi ||_{\infty} < \infty$.

A direct calculation yields
$$2\bar{\delta\psi}'(m) = \int{(\psi'(-x+m) + \psi'(x+m))\nu(dx)} - \int{(\psi_c'(-x+m) + \psi_c'(x+m))\nu_c(dx)}.$$

We introduce the family of measures $(\nu^s)_{s \in [0,1]}$, defined by
$$\nu^s(dx) := \frac{1}{Z}\exp(-\psi_c(-x+m) - \psi_c(x+m) -s\delta\psi(-x+m) -s\delta\psi(x+m))dx.$$
This family interpolates between $\nu^0 = \nu_c$ and $\nu^1 = \nu$. By the mean-value theorem, there exists $s \in [0,1]$ such that
\begin{align}
2\bar{\delta\psi}'(m) &= \frac{d}{ds}\int{(\psi_c'(-x+m) + \psi_c'(x+m) + s\delta\psi'(-x+m) + s\delta\psi'(x+m))\nu^s(dx)} \notag \\
&= \int{(\delta\psi(-x+m) + \delta\psi(x+m))\nu^s(dx)} \notag \\
& \hspace{5mm} - \cov_{\nu^s} \left( \psi_c'(-x+m) + \psi_c'(x+m), \delta\psi(-x+m) + \delta\psi(x+m) \right) \notag \\
& \hspace{5mm} - \cov_{\nu^s} \left( s\delta\psi'(-x+m) + s\delta\psi'(x+m), \delta\psi(-x+m) + \delta\psi(x+m) \right) \notag
\end{align}

The first and third term on the right-hand side of this equation can be bounded uniformly in $m$ by using the assumption that $\delta\psi$ and $\delta\psi'$ are bounded. For the second term, we also use these bounds, as well as the asymmetric Brascamp-Lieb inequality of Lemma \ref{as_bl} to show that
\begin{align}
\cov_{\nu^s} & \left( \psi_c'(-x+m) + \psi_c'(x+m), \delta\psi(-x+m) + \delta\psi(x+m) \right) \notag \\
&\leq C \underset{x}{\sup} \hspace{1mm} \left| \frac{\psi_c''(-x+m) - \psi_c''(x+m)}{\psi_c''(-x+m) + \psi_c''(x+m)} \right|\int{|-\delta\psi'(-x+m) + \delta\psi'(x+m)|\nu^s(dx)} \notag \\
&\leq C, \notag
\end{align}
which finishes the proof of $||\bar{\delta\psi}'||_{\infty} < \infty$. This concludes the proof of Lemma \ref{structure}.

\end{proof}

Finally, we prove Lemma \ref{lem_convexification}, which is the last remaining step.

\begin{proof} [Proof of Lemma \ref{lem_convexification}]
We define
\begin{equation} \label{def_varphi}
\varphi(m) := \underset{\sigma \in \R}{\sup} \hspace{1mm} \left( \sigma m - \log \int_{\R}{\exp(\sigma x - \psi(x))dx} \right).
\end{equation}
It is the Legendre transform of the function
\begin{equation} 
\varphi^*(\sigma) := \log \int{\exp(\sigma x - \psi(x))dx}.
\end{equation}

\begin{equation}
\mu^{\sigma}(dx) = \exp(\sigma x - \psi(x) - \varphi^*(\sigma))dx
\end{equation}

\begin{thm}[Local Cram\'er theorem, Menz-Otto 2011]
Let $$\psi_K(m) := -\frac{1}{K} \log \left( \int_{X_{K,m}}{\exp(- \sum \psi(x))dx} \right).$$ 
If $\psi$ is a bounded pertubation of a uniformly convex potential, we have
$$\left| \psi_K''(m) - \varphi''(m) \right| \leq \frac{C}{K}\varphi''(m)$$
uniformly in $m \in \R$.
\end{thm}

Since the proof of this result is quite long, we will not reproduce it here, and refer the interested reader to \cite{MO}.

We apply this theorem, and since $\textit{R}^M\psi = \psi_{2^M}$, for $M$ large enough we have, uniformly in $m$,
$$\textit{R}^M\psi''(m) \geq \frac{1}{2}\varphi''(m).$$
Direct calculation on expression (\ref{def_varphi})
$$\varphi''(m) = \frac{1}{s(\sigma_m)^2},$$
where 
$$s(\sigma)^2 = \int{(x-m)^2\mu^{\sigma}(dx)},$$

$$\mu^{\sigma}(dx) := \frac{1}{Z}\exp(\sigma x - \psi(x))dx,$$
and $\sigma$ is the unique real number such that $\int{x \mu^{\sigma}(dx)} = m$. 

The measures $\mu^{\sigma}$ satisfy a Poincar\'e inequality with constant independent of $\sigma$, therefore we can show that $s(\sigma)^2$ is bounded above independently of $\sigma$ :
$$s(\sigma)^2 \leq \frac{1}{\rho}\int{|\nabla x |^2 \mu^{\sigma}(dx)} = \frac{1}{\rho},$$
and the uniform convexity of $\textit{R}^M\psi''$ follows.

To show that $\textit{R}^M\psi$ is p-convex, it is therefore enough to show that 
\begin{equation} \label{p_convexity_coarse}
\varphi''(m) \geq C|m - m_0|^{p-2}
\end{equation}
for some $C > 0$ and $m_0 \in \R$. Let 
$$m_0 := \int{x \mu^0(dx)}.$$
Since, by the usual properties of the Legendre transform, the real number $\sigma_m$ such that $\varphi(m) = m\sigma_m - \varphi^*(\sigma_m)$ is given by $\varphi'(m) = \sigma_m$, we have $\varphi'(m_0) = 0$, and the unique minimum of $\varphi$ is reached at $m_0$. 
Since $\mu^0$ satisfies p-LSI($\rho$) for some $\rho > 0$ (to show this, use the p-convexity of $\psi_c$ and the Holley-Stroock lemma), applying Proposition \ref{laplace}, we have
$$\frac{1}{\int{\exp(-\psi(x))dx}}\int{\exp(\sigma x - \psi(x))dx} \leq \exp \left(\sigma\int{x\mu^0(dx)} + \frac{|\sigma|^q}{\rho(q-1)} \right)$$
and therefore 
$$\varphi^*(\sigma) \leq \varphi^*(0) + \sigma m_0 + \frac{|\sigma|^q}{\rho(q-1)}.$$
We then have
\begin{align} \label{aer}
\varphi(m) &= \underset{\sigma \in \R}{\sup} \hspace{1mm} \left( \sigma m - \varphi^*(\sigma) \right) \notag \\
&\geq \underset{\sigma \in \R}{\sup} \hspace{1mm} \left( \sigma m - \varphi^*(0) -  \sigma m_0 - \frac{|\sigma|^q}{\rho(q-1)} \right) \notag \\
&= \varphi(m_0) + \underset{\sigma \in \R}{\sup} \hspace{1mm} \left( \sigma (m - m_0) - \frac{|\sigma|^q}{\rho(q-1)} \right) \notag \\
&= \varphi(m_0) + c|m - m_0|^p
\end{align}

where $c$ is a positive constant which only depends on $\rho$ and $p$. We then consider $f(m) = (m - m_0)\varphi'(m) - \varphi(m)$. Since $\varphi''$ is positive, $f$ reaches its minimum at $m_0$, so that for all $m \in \R$ we have $(m - m_0)\varphi'(m) - \varphi(m) \geq -\varphi(m_0)$, and therefore, using (\ref{aer}) and the fact that $\varphi'$ is increasing, we get 
\begin{equation}
|\varphi'(m)| \geq c|m - m_0|^{p-1}.
\end{equation}
To study the behavior of $\varphi''$, we shall now look at $\varphi^{(3)}$. An explicit calculation shows that

\begin{align}
\varphi^{(3)}(m) &= \frac{d}{dm} \left(\int{(x - m)^2\mu^{\sigma}(dx)}\right)^{-1} \notag \\
&= \frac{d\sigma}{dm} \times \frac{d}{d\sigma} \left(\int{(x - m)^2\mu^{\sigma}(dx)}\right)^{-1} \notag \\
&= -\left(\int{(x - m)^3\mu^{\sigma}(dx)}\right)\left(\int{(x - m)^2\mu^{\sigma}(dx)}\right)^{-3}
\end{align}
so that $\varphi^{(3)}(m) = 0$ iff $\int{(x - m)^3\mu^{\sigma}(dx)} = 0$. But we have
$$\frac{d}{d\sigma}\int{(x - m)^3\mu^{\sigma}(dx)} = \int{(x-m)^4\mu^{\sigma}(dx)} > 0$$
so that $\int{(x - m)^3\mu^{\sigma}(dx)}$ is a strictly increasing function, and cancels for at most one value of $m$. Therefore there exists some $m_1 \in \R$ such that $\varphi^{(3)}$ has constant sign on $(m_1, +\infty)$. Without loss of generality, we can assume $m_1 > m_0$. We consider two cases : 

If $\varphi^{(3)}$ is non-negative on $(m_1, +\infty)$, then for any $\alpha \in [0, 1]$ the function $(m - m_0)\varphi''(m) - \alpha \varphi'(m)$ is increasing on $(m_1, +\infty)$. Moreover, since $m_1 > m_0$, $\varphi'(m_1) > 0$, and if we take $\alpha = \min (1, \frac{(m_1 - m_0)\varphi''(m_1)}{\varphi'(m_1)})$, this function is nonnegative at $m = m_1$. Therefore, for any $m \in (m_1, +\infty)$, we have
\begin{align}
\varphi''(m) &\geq \alpha \frac{\varphi'(m)}{m - m_0} \notag \\
&\geq c|m - m_0|^{p-2}.
\end{align}

If $\varphi^{(3)}$ is negative on $(m_1, +\infty)$, then $\varphi''$ is decreasing, and since it is bounded below by a positive constant, it converges to some positive constant $\lambda > 0$ in $+\infty$. We then have
$$\varphi'(m) = \int_{m_0}^m{\varphi''(s)dx} \underset{m \rightarrow +\infty}{\sim} \lambda m.$$
But since we know that $\varphi'(m) \geq c|m - m_0|^{p-1}$ with $p > 2$, this is a contradiction, so $\varphi^{(3)}$ must be non-negative on $(m_1, +\infty)$. Therefore we have 
$$\varphi''(m) \geq c|m - m_0|^{p-2}$$
for all $m > m_1$. With the same reasoning, we can show that $\varphi''(m) \geq c|m - m_0|^{p-2}$ for all $m < m_2$ for some $m_2 < m_0$. But since $\varphi''$ is bounded below by a strictly positive constant, if we take $c$ small enough, we also have $\varphi''(m) \geq c|m - m_0|^{p-2}$ for all $m \in [m_2, m_1]$, and therefore (\ref{p_convexity_coarse}) holds. This concludes the proof of Lemma \ref{lem_convexification}.

\end{proof}

\section{An application to Kawasaki dynamics}

There are many results on convergence to equilibrium in relative entropy for various dynamics in the literature. Theorem \ref{thm_tal} says that, when we have such a convergence and if the invariant measure is the canonical ensemble $\mu_{N,m}$, then we also have convergence in the Wasserstein distance $W_p$. An example of such a dynamic with conservation law is given by the Kawasaki dynamic on $R^N$ :
$$dX_t = -A\nabla H(X_t)dt + \sqrt{2A}dB_t$$
where $B_t$ is a Brownian motion on $R^N$ and $A$ is the discrete Laplacian on $R^N$, that is
$$A_{i,j} := 2\delta_{i,j} - \delta_{i,j+1} - \delta_{i,j-1}.$$
If we assume that the law of the initial value $X_0$ is absolutely continuous with respect to $\mu = \exp(-H)dx$, then the law $f_t$ of $X_t$ satisfies (in a weak sense) the PDE
$$\frac{\partial f_t}{\partial t} = \nabla \cdot (A\nabla f_t \mu).$$
Since this dynamic conserves the average $\sum x_i$, we restrict it to the hyperplane $\left\{\sum x_i = Nm \right\}$. It is a consequence of the LSI proved in \cite{MO} that, when $H(x) = \sum \psi(x_i)$ with $\psi$ a bounded perturbation of a uniformly convex potential, the entropy satisfies the bound
$$\Ent_{\mu}(f_t) \leq \exp(-\rho t/N^2)\Ent_{\mu}(f_0),$$
and the order of magnitude $t/N^2$ is optimal. The following result is then an immediate consequence of this bound and our results :

\begin{prop}
Assume that $f_t$ is the law of a solution of the Kawasaki dynamics with initial condition $f_0\mu$. Assume that the single-site potential satisfies (\ref{assumption_potential}). Then we have convergence to equilibrium for $W_p$, in the following sense :
$$W_p^p(f_t\mu, \mu) \leq C\exp(-\rho t/N^2)\Ent_{\mu}(f_0),$$
with constants $C$ and $\rho$ independent of the dimension $N$ and the mean spin $m$, and the $\ell^p$ distance.
\end{prop}

\vspace{1cm}

\appendix

{\Large \textbf{Appendix}}

\section{Standard criteria for modified LSI}

In this section, we state some standard criteria for a measure to satisfy a modified LSI. These criteria are respectively the natural equivalents of the Bakry-Emery theorem, the tensorization principle and the Holley-Stroock lemma for classical the LSI.

\begin{thm} \label{be}
Let $V$ be a uniformly p-convex function with constant $\rho$ on $\mathbb{R}^N$, that is for any $x$, $y \in \mathbb{R}^N$ and $t \in [0,1]$, we have
$$V(tx + (1-t)y) \leq tV(x) + (1-t)V(y) - \rho\frac{t(1-t)}{p}||x - y||_p^p.$$
Then the probability measure $\mu(dx) = \frac{1}{Z}\exp(-V(x))dx$ satisfies $p-LSI((\rho/q)^{q-1})$.
\end{thm}

For a proof of this result, we refer to \cite{BL}.

\begin{exple}
$\mu(dx) = \exp(-||x||_p^p)dx$ satisfies p-LSI(c) for some $c > 0$.
\end{exple}

\begin{prop}
If $V : \R \rightarrow \R$ satisfies $V''(x) \geq c(p-1)|x|^{p-2}$, then $V$ is p-convex with constant $c$. 
\end{prop}

\begin{rmq}
This is not a necessary condition. $x \rightarrow (x-1)^4$ is 4-convex with constant $4$, yet we do not have $12(x-1)^2 \geq 12x^2$.
\end{rmq}

\begin{prop} \label{tensor}
If $\mu$ (resp. $\nu$) is a probability measure on $X_1$ (resp. $X_2$) satisfying $p-LSI(\rho_1)$ (resp. $p-LSI(\rho_2)$), then $\mu \otimes \nu$ satisfies $p-LSI(\min(\rho_1, \rho_2))$.

\end{prop}

\begin{proof}
It is proven in the same way as for the usual LSI, by using the inequality
$$\operatorname{Ent}_{\mu \otimes \nu}(f^q) \leq \int_{X_2}{\operatorname{Ent}_{\mu}(f(\cdot, x_2)^q)\nu(dx_2)} + \int_{X_1}{\operatorname{Ent}_{\nu}(f(x_1, \cdot)^q)\mu(dx_1)}$$
and applying the p-LSI for each measure. See for example \cite{L}, Proposition 5.6 for a proof of this inequality.
\end{proof}

\begin{prop} \label{hol_str}
If $\mu$ satisfies $p-LSI(\rho)$ and $\psi$ is a bounded function, then the probability measure $\nu = \frac{1}{Z}\exp(\psi)d\mu$ satisfies $p-LSI(e^{2\operatorname{osc}(\psi)}\rho)$, where $\operatorname{osc}(\psi) = \sup \psi - \inf \psi$.
\end{prop}

\begin{proof} This
 is the analogue of the Holley-Stroock lemma for the usual LSI, and we can prove it in the same way. The identity (valid for any probability measure $\mu$)
$$\operatorname{Ent}_{\mu}(f) = \underset{t \geq 0}{\inf} \int_X{f \log f - t\log t + (t - f)(1 + \log t)d\mu}$$
implies that 
$$\operatorname{Ent}_{\nu}(f^q) \geq \exp(\operatorname{osc}(\psi))\operatorname{Ent}_{\mu}(f^q).$$
It is also easy to show that
$$\int{||\nabla f||_q^qd\mu} \leq \exp(\operatorname{osc}(\psi))\int{||\nabla f||_q^qd\nu},$$
so that, if $\mu$ satisfies p-LSI($\rho$), then $\nu$ satisfies p-LSI($e^{2\operatorname{osc}(\psi)}\rho$).
\end{proof}

\begin{prop} \label{laplace}
If a probability measure $\mu$ on $\R^n$ (endowed with the $L^p$ norm) satisfies $p-$LSI($\rho$), then for any 1-Lipschitz function $f$ such that $\int{f d\mu} = 0$, we have $\int{e^{\lambda f}d\mu} \leq \exp\left(\frac{\lambda^q}{\rho(q-1)} \right)$ for all $\lambda \geq 0$.
\end{prop}

\begin{proof}
Let $f$ be a smooth 1-Lipschitz function on $X$ for the $||\cdot||_p$ norm, with mean 0, and $$H(\lambda) := \int{\exp(\lambda f - c\lambda^q||f||_{lip}^q)d\mu}.$$ Then
\begin{align}
\frac{d}{d\lambda}H(\lambda) &= \int{(f - qc\lambda^{q-1}||f||_{lip}^q)\exp(\lambda f - c\lambda^q||f||_{lip}^q)d\mu} \notag \\
&=\frac{1}{\lambda} \int{(\lambda f - cq\lambda^q||f||_{lip}^q)\exp(\lambda f - c\lambda^q||f||_{lip}^q)d\mu} \notag \\
&=\frac{1}{\lambda} \int{(\lambda f - c\lambda^q||f||_{lip}^q)\exp(\lambda f - c\lambda^q||f||_{lip}^q)d\mu} + \frac{c(1-q)}{\lambda}\lambda^q||f||_{lip}^q H(\lambda) \notag \\
&= \frac{1}{\lambda}\operatorname{Ent}_{\mu}(\exp(f - c\lambda^q||f||_{lip}^q)) + \frac{c(1-q)}{\lambda}\lambda^q||f||_{lip}^q H(\lambda) \notag 
\end{align}

We can use the assumption that $\mu$ satisfies the p-LSI with parameter $\rho$ under the form (\ref{mlsi2}) to bound the entropy term, and we obtain
\begin{equation}
\frac{d}{d\lambda}H(\lambda) \leq \frac{1}{\lambda \rho} \int{\lambda^q||\nabla f||_q^q\exp(f - c\lambda^q||f||_{lip}^q)d\mu} + \frac{c(1-q)}{\lambda}\lambda^q||f||_{lip}^q H(\lambda).
\end{equation}

Since we assumed $f$ to be 1-Lipschitz for the $L^p$ norm, $||\nabla f||_q \leq ||f||_{lip}$ almost everywhere, and therefore
\begin{align}
\frac{d}{d\lambda}H(\lambda) &\leq \frac{1}{\lambda \rho} \int{\lambda^q||f||_{lip}^q\exp(\lambda f - c\lambda^q||f||_{lip}^q)d\mu} + \frac{c(1-q)}{\lambda}\lambda^q||f||_{lip}^q H(\lambda) \notag \\
&= \left( \frac{1}{\rho} + c(1-q) \right) \lambda^{q-1}||f||_{lip}^qH(\lambda). \notag
\end{align}
Taking $c = 1/\rho(q-1)$, we get $\frac{d}{d\lambda}H(\lambda) \leq 0$, therefore $H(\lambda) \leq H(0) = 1$ for all $\lambda \geq 0$, so that 
$$\int{\exp(\lambda f)d\mu} \leq \exp\left(\frac{\lambda^q||f||_{lip}^q}{\rho(q-1)} \right)$$
for all $\lambda \geq 0$, which implies the desired result.

\end{proof}

\vspace{5mm}
\underline{\textbf{Acknowledgements}}: I would like to thank Emmanuel Boissard, Nathael Gozlan, Georg Menz and C\'edric Villani for discussions on this topic. I would also like to thank the (anonymous) referee for pointing out a mistake in a previous version of the proof of Lemma 2.3, and for his suggestions on improvements of the presentation.


\vspace{1cm}

\begin{thebibliography}{99}

\bibitem{BGL}
	Bobkov, S. G., Gentil, I. and Ledoux, L.,
	Hypercontractivity of Hamilton-Jacobi equations
	\textit{J. Math. Pures Appl}. \textbf{80}, 669-696 (2001).

\bibitem{BL}
	Bobkov, S. G. and Ledoux, M.,
	From Brunn-Minkowski to Brascamp-Lieb and to logarithmic Sobolev inequalities,
	\textit{Geom. func. anal.} \textbf{10} (2000), 1028-1052.
	
\bibitem{BZ}
	Bobkov, S. G. and Zegarlinski, B.,
	Entropy Bounds and Isoperimetry,
	Memoirs of the AMS (2005)
	
\bibitem{F}
	Fathi, M.,
	A two-scale approach to the hydrodynamic limit, part II : local Gibbs behavior.
	\textit{ALEA}, \textbf{80}, vol. 2, 625-651 (2013)

\bibitem{Go}
	Gozlan, N.,
	A Characterization of Dimension-Free Concentration in Terms of Transportation Inequalities
	\textit{Ann. Probab.} \textbf{37}, Number 6 (2009), 2480-2498.
	
\bibitem{Gr}
	L. Gross,
	Logarithmic Sobolev inequalities,
	\textit{Amer. J. Math.} 97, 1061-1083 (1975).
	
\bibitem{GRS}
	Gozlan, N., Roberto, C. and Samson, P.,
	Characterization of Talagrand's transport-entropy inequality in metric spaces.
	\textit{Annals of Probability}, \textbf{41} (2013), no. 5, 3112--3139.

\bibitem{GOVW}
	Grunewald, N., Otto, F., Villani, C. and Westdickenberg, M. G.,
	A two-scale approach to logarithmic Sobolev inequalities and the hydrodynamic limit.
	\textit{Ann. Inst. H. Poincar\'e Probab. Statist}. 
	45 (2009), 2, 302--351.
	
\bibitem{L}
	Ledoux, M.,
	The Concentration of Measure Phenomenon,
	AMS, Math. Surveys and Monographs, \textbf{89}, Providence, Rhode Island, 2001.
	
\bibitem{L2}
	Ledoux, M.,
	Logarithmic Sobolev inequalities for spin systems revisited,
	1999, http://citeseerx.ist.psu.edu/viewdoc/summary?doi=10.1.1.36.4917
		
\bibitem{MO}
	Otto, F. and Menz, G.,
	Uniform logarithmic Sobolev inequalities for conservative spin systems with super-quadratic single-site potential.
	\textit{Ann. Probab.}, \textbf{41}, Number 3B (2013), 2182-2224.
	
\bibitem{OV}
	Otto, F. and Villani, C.,
	Generalization of an Inequality by Talagrand and Links with the Logarithmic Sobolev Inequality,
	\textit{J. Funct. Analysis}, \textbf{243} (2007), pp. 121-157.
	

\bibitem{V}
	Villani C.,
	Optimal Transport, Old and New.
	\textit{Grundlehren der mathematischen Wissenschaften},
	Vol. 338, Springer-Verlag, 2009.

\end{thebibliography}
\end{document}